\documentclass[reqno]{amsart}

\usepackage{amsmath}
\usepackage{amssymb}
\usepackage{amsfonts}
\usepackage{amsthm}
\usepackage[foot]{amsaddr}
\usepackage{mathtools}
\mathtoolsset{%
}
\usepackage{mathscinet}

\usepackage[utf8]{inputenc}
\usepackage[T1]{fontenc}

\usepackage{libertine}
\usepackage[libertine,libaltvw,cmintegrals,varbb]{newtxmath}

\usepackage{dsfont}

\usepackage{mathrsfs}

\usepackage[%
cal=cm,
]
{mathalfa}

\usepackage[dvipsnames,svgnames]{xcolor}
\colorlet{MyBlue}{DodgerBlue!60!Black}
\colorlet{MyGreen}{DarkGreen!85!Black}

\usepackage[font=small,labelfont=bf]{caption}
\usepackage{subfigure}
\usepackage{tikz}
\usetikzlibrary{calc,patterns}

\usepackage{acronym}
\usepackage{booktabs}       
\usepackage{latexsym}
\usepackage{paralist}
\usepackage{wasysym}
\usepackage{xspace}

\usepackage[authoryear,compress,comma]{natbib}

\usepackage{hyperref}
\hypersetup{
colorlinks=true,
linktocpage=true,
pdfstartview=FitH,
breaklinks=true,
pdfpagemode=UseNone,
pageanchor=true,
pdfpagemode=UseOutlines,
plainpages=false,
bookmarksnumbered,
bookmarksopen=false,
bookmarksopenlevel=1,
hypertexnames=true,
pdfhighlight=/O,
urlcolor=MyBlue!60!black,linkcolor=MyBlue!70!black,citecolor=DarkGreen!70!black, 
pdftitle={},
pdfauthor={},
pdfsubject={},
pdfkeywords={},
pdfcreator={pdfLaTeX},
pdfproducer={LaTeX with hyperref}
}

\numberwithin{equation}{section}  
\usepackage[sort&compress,capitalize,nameinlink]{cleveref}
\crefname{app}{Appendix}{Appendices}

\crefrangeformat{equation}{\upshape(#3#1#4)\textendash(#5#2#6)}

\newcommand{\R}{\mathbb{R}}

\newcommand{\N}{\mathbb{N}}

\DeclareMathOperator{\ex}{\debug{\mathbb{E}}}

\DeclareMathOperator{\prob}{\debug{\mathbb{P}}}

\DeclarePairedDelimiter{\norm}{\lVert}{\rVert}

\DeclarePairedDelimiterX{\braket}[2]{\langle}{\rangle}{#1,#2}

\DeclarePairedDelimiterX{\inner}[2]{\langle}{\rangle}{#1,#2}
\DeclarePairedDelimiterX{\setdef}[2]{\{}{\}}{#1:#2}

\DeclarePairedDelimiterXPP{\probof}[1]{\prob}{(}{)}{}{%

#1}

\DeclarePairedDelimiterXPP{\exof}[1]{\ex}{[}{]}{}{%

#1}

\usepackage[textwidth=30mm]{todonotes}

\newcommand{\debug}[1]{#1}

\theoremstyle{plain}
\newtheorem{theorem}{Theorem}
\newtheorem{corollary}[theorem]{Corollary}
\newtheorem*{corollary*}{Corollary}
\newtheorem{lemma}[theorem]{Lemma}
\newtheorem{proposition}[theorem]{Proposition}

\theoremstyle{definition}
\newtheorem{definition}[theorem]{Definition}
\newtheorem*{definition*}{Definition}

\newtheorem*{hypothesis*}{Hypothesis}

\theoremstyle{remark}
\newtheorem{remark}{Remark}
\newtheorem*{remark*}{Remark}
\newtheorem*{notation*}{Notational remark}
\newtheorem{example}{Example}

\numberwithin{theorem}{section}
\numberwithin{remark}{section}
\numberwithin{example}{section}

\title
[Continuous Patrolling and Hiding Games]
{Continuous Patrolling and Hiding Games}

\author
[T.~Garrec]
{Tristan Garrec}
\address{TSE-R, Toulouse School of Economics,  Manufacture des Tabacs, 21 Allée de Brienne, 31015 Toulouse Cedex 6, France }
\email{tristan.garrec@ut-capitole.fr}


\newcommand{\A}{\mathcal{A}}
\newcommand{\W}{\mathcal{W}}

\newcommand{\n}{\mathcal{N}}

\newcommand{\Proba}{\mathbb{P}}
\newcommand\eucl[1]{\left\lVert#1\right\rVert_2}

\begin{document}
\maketitle

\begin{abstract}
We present two zero-sum games modeling situations where one player attacks (or hides in) a finite dimensional nonempty compact set, and the other tries to prevent the attack (or find him). The first game, called patrolling game, corresponds to a dynamic formulation of this situation in the sense that the attacker chooses a time and a point to attack and the patroller chooses a continuous trajectory to maximize the probability of finding the attack point in a given time. Whereas the second game, called hiding game, corresponds to a static formulation in which both the searcher and the hider choose simultaneously a point and the searcher maximizes the probability of being at distance less than a given threshold of the hider.
\end{abstract}

\section{Introduction}

To ensure the security of vulnerable facilities, a planner may deploy either dynamic or static security devices. Some examples of dynamic security devices are: soldiers or police officers patrolling the streets of a city, robots patrolling a shopping mall, drones flying above a forest to detect fires, or naval radar systems signaling the detection of an enemy ship. On the other hand, security guards positioned in the rooms of a museum, security cameras scrutinizing subway corridors, or motion detectors placed in a house are some examples of static security devices.
The paradigm we adopt is the one of an adversarial threat, hence we propose a game theoretical approach to these security problems. Some game theoretic security systems are already in use, for example in the Los Angeles international airport, see \citep{pitaetal2008}, and in some ports of the United States, see \citep{shiehetal2012}.

Motivated by the examples given above, we study two zero-sum games in which a player (the patroller or searcher) aims to detect another player (the attacker or hider). \textit{Patrolling games} model dynamic security devices. In these games, the patroller moves continuously in a search space with bounded speed. The attacker chooses a point in the search space and a time to attack it. The attack takes a certain duration to be successful (think of a terrorist needing time to set off a bomb). The patroller wins if and only if she detects the attack before it succeeds. The case of static security devices is modeled by \textit{hiding games}, in which both the searcher and the hider simultaneously deploy at some points in a search space, the searcher wins if and only if the hider lies within her detection radius. We provide links between patrolling and hiding games, and show how patrolling games reduce to hiding games if the attack duration is zero, that is the patroller has to detect the attack at the exact time it occurs, or if the patroller is alerted of the attack point when the attack begins.

\subsection{Contribution}

For patrolling games, we prove that the value always exists, and obtain a decomposition result to lower bound the value as well as a general upper bound. We then study patrolling games on networks. We compute the value as well as optimal strategies for the class of Eulerian networks. The special network composed of two nodes linked by three parallel arcs is also examined and bounds on the value are computed. Lastly, we study patrolling games on $\R^2$ and obtain an asymptotic expression for the value as the detection radius of the patroller goes to $0$.

For hiding games, we first focus on a particular class of strategies for both the searcher and the hider called "equalizing". These strategies have the property that if there exists one, then it is optimal for both players. Our main result regarding hiding games is an asymptotic formula for the value when the search space has positive Lebesgue measure. A counterexample based on a Cantor-type set showing that this last result cannot be extended to compact sets with zero Lebesgue measure is also presented.

Finally, we discuss some elementary properties of monotonicity and continuity of the value function of continuous patrolling and hiding games.

\subsection{Related literature}
Continuous patrolling and hiding games belong to the literature on search and security games, consult \citep{Hoh:JORSJ2016} for a survey. These games have their source in search theory, a field of operations research whose origins can be found in the works of \cite{koopman1956a,koopman1956b,koopman1957}. The use of game theory in a search and security context goes back to the famous book of \cite{morsekimball1951}. 

\subsubsection{Patrolling games}
Patrolling games were introduced by \cite{alpernetal2011} in a discrete setting, that is the patroller visits nodes of a graph, where the attacker can strike, at discrete times. A companion article \citep{alpernetal2016b} is dedicated to the resolution of patrolling a discrete line. The idea of investigating a continuous version of patrolling games is suggested in \citep{alpernetal2011}. In \citep{alpernetal2016a}, the authors solve the continuous patrolling game played on the unit interval. Several papers in the field of search games have dealt with the transcription of discrete models to continuous ones, as in \citep{rucklekikuta2000} and \citep{ruckle1981}.

Other game theoretic models involving a patroller and an attacker can be found in \citep{basilicoetal2012,basilicoetal2015}, in which the authors design algorithms to solve large instances of Stackelberg patrolling security games on graphs. \cite{linetal2013,linetal2014} use linear programming and heuristics to study a large class of patrolling problems on graphs, with nodes having different values. \cite{zoroaetal2012} study a patrolling game with a mobile attacker on a perimeter. 

Continuous patrolling games are closely related to search games with an immobile hider, introduced in the seminal book of \cite{Isa:Wiley1996} and developed in the monographs of \cite{Gal:AP1980} and \cite{AlpGal:Kluwer2003}. In these games, a searcher intends to minimize the time necessary to find a hider. Search games have been extensively studied, let us mention \citep{Gal:SJCO1979,AlpBasGal:IJGT2008,DagGal:N2008} for search games on a network. See also \citep{Bos:SIAMJADM1984,pavlovic1993} for the special network consisting of two nodes linked by three parallel arcs, for which the solution is surprisingly complicated.

\subsubsection{Hiding games}

The first published example of a hiding game goes back to \cite{galeglassey1974}. The proposers gave a solution to the problem of hiding in the unit disc when the detection radius is $r=1/2$. Later, \cite{ruckle1983} considered in his book several examples of hiding games (hiding on a sphere, hiding in a disc, among others). Computing the value of hiding games is in general a very difficult task. \cite{danskin1990} improved substantially the resolution of the hiding game played on a disc, he called the cookie-cutter game. However the solution is not complete for small values of $r$ and no progress has been made since then, see also \citep{alpernetal2013} and \citep{washburn2014}. 
Hiding games in a discrete setting, i.e., when the search space is a graph, have been studied by \cite{bishopevans2013}.

Games in which the payoff is the distance between the searcher and the hider, introduced by \cite{karlin1959}, have been extensively studied, consult \citep{ibragimovsatimov2012} and references therein. Although these games resemble hiding games to a certain extent, the lack of continuity of the payoff function in hiding games makes their analysis much more involved.

Finally, ambush games can be seen as hiding games in which the players select not one, but several points in the unit interval, consult \citep{zoroaetal1999,bastonkikuta2004,bastonkikuta2009} for further details.

\subsection{Organization of the paper}

The paper is organized as follows. In \cref{sec_models} the models are formally presented. \cref{sec_patrol} is dedicated to patrolling games, and \cref{sec_hiding} is dedicated to hiding games. Finally in \cref{sec_propertiesvalue} we give some properties of the value function of continuous patrolling and hiding games. Proofs that are not included in the body of the paper are postponed to the \hyperlink{app:appendix_patrol}{Appendix of Chapter 4}.

\subsection{Notations}

In all the article, $\R^n$ is endowed with a norm denoted $\Vert \cdot \Vert$, which induces a metric $d$. For all $x\in\R^n$ and $r>0$, the closed ball of center $x$ with radius $r$ is denoted $B_r(x) = \{y\in \R^n \ | \ \Vert x - y \Vert \leq r \}$. For all Lebesgue measurable set $B\subset\R^n$, $\lambda(B)$ denotes its Lebesgue measure. Finally, $\lambda(B_r)$ denotes the Lebesgue measure of any ball of radius $r$. Let $X$ be a topological space, the set of Borel probability measures on $X$ is denoted $\Delta(X)$, and the set of probability measures on $X$ with finite support is denoted $\Delta_f(X)$.

\section{The models}
\label{sec_models}
\subsection{Patrolling games}

In a patrolling game two players, an attacker and a patroller, act on a set $Q$ called the search space, which is assumed to be a nonempty compact subset of $\R^n$. An example of this could be a metric network, as in \cref{subsec_network}. The attacker chooses an attack point $y$ in $Q$ and a time to attack $t$ in $\R_+$. The patroller walks continuously in $Q$ with speed at most $1$. When the attack occurs at time $t$ and point $y$, the patroller has a time limit $m\in\R_+$ to be at distance at most $r\in\R_+$ of the attack point $y$. In this case she detects the attack and wins, and otherwise she does not. Thus, $m$ represents the time needed for an attack to be successful, and $r$ represents the detection radius of the patroller.

A patrolling game is thus a zero-sum game given by a triplet $(Q,m,r)$. The attacker's set of pure strategies is $\A = Q\times\R_+$. An element of $\A$ is called an attack. The patroller's set of pure strategies is \[\W = \{w:\R_+\to Q \ | \ w \text{ is } 1\text{-Lipschitz continuous}\}.\] An element of $\W$ is called a walk. $\W$ is endowed with the topology of compact convergence (consult the proof of \cref{prop_existence_value} in the \hyperlink{app:appendix_patrol}{Appendix of Chapter 4} and \citep{munkres2000} for details).

The payoff to the patroller is given by 
\[g_{m,r}(w,(y,t))=\left\{
    \begin{array}{ll}
1 &\text{ if } d(y,w([t,t+m]))\leq r \\
0 &\text{ otherwise},
    \end{array}
\right.
\]
where $w([t,t+m])=\{w(\tau) \ | \ \tau\in [t,t+m]\}.$

\subsection{Hiding games}

In a hiding game two players, a searcher and a hider, act on a search space $Q$, which is again assumed to be a nonempty compact subset of $\R^n$. Both players choose a point in $Q$. The searcher has a detection radius $r\in\R_+$. She finds the hider if and only if the two points are at distance at most $r$.

Hence, a hiding game is a zero-sum game given by a couple $H=(Q,r)$. The set of pure strategies of both players, the searcher and hider, is $Q$. The payoff to the searcher is given by 
\[h_r(x,y)=\left\{
    \begin{array}{ll}
1 &\text{ if } \Vert x-y\Vert \leq r \\
0 &\text{ otherwise}.
    \end{array}
\right.\]

\subsection{Links between patrolling and hiding games}

Hiding games can be interpreted as two possible variants of patrolling games.

In the first variant, hiding games are considered as a particular class of patrolling games in which the attack duration $m$ is equal to $0$. Indeed, consider a hiding game $H=(Q,r)$ and a patrolling game $P=(Q,0,r)$. In $P$, for all $w\in\W$ and $(y,t)\in \A$ the payoff to the patroller is
\[g_{0,r}(w,(y,t))=\left\{
    \begin{array}{ll}
1 &\text{ if } \Vert w(t) - y \Vert \leq r \\
0 &\text{ otherwise}.
    \end{array}
\right.\]
Let $x\in Q$ be a strategy of the searcher in the hiding game $H$, and consider the constant walk $w_x : t\mapsto x$ of the patroller in $P$. Then for any strategy $(y,t)$ of the attacker in $P$, the payoff $g_{0,r}(w_x,(y,t))$ equals $1$ if $\Vert x - y \Vert \leq r$ and $0$ otherwise. Since the payoff is time-independent, any quantity guaranteed by the searcher in $H$ is guaranteed by the patroller in $P$. Conversely, let $y\in Q$ be a strategy  of the hider in $H$ and consider the strategy $(y,0)$ of the attacker in $H$. Then for any strategy $w$ of the patroller in $P$, the payoff is $1$ if $\Vert w(0) - y \Vert \leq r$ and $0$ otherwise. Again the payoff is time-independent and any quantity guaranteed by the hider in $H$ is guaranteed by the attacker in $P$. Thus, since $H$ and $P$ have a value (see \cref{prop_existence_value,prop_existence_value_hide}), the values of these two games are the same.

The second interpretation is as follows. \cite{alpernetal2011} suggest the study of patrolling games in which the patroller may be informed of the presence of the attacker.
Suppose that the patroller is informed of the attack point when the attack occurs. Suppose also the search space $Q$ is convex. The detection radius $r$ is set to $0$ for simplicity. The payoff of this game is
\[g_{m,0}(w,(y,t))=\left\{
    \begin{array}{ll}
1 &\text{ if } y\in w([t,t+m])\\
0 &\text{ otherwise}.
    \end{array}
\right.\]
This patrolling game with signals is denoted $P'$. In $P'$, if the patroller's strategy is to choose a point and not move until the attack, then go to the attack point in straight line when she is alerted, the attacker is time-indifferent. In particular, the attacker has a best response in the set of attacks occurring at time $0$. Symmetrically, if the attack occurs at time $0$, the patroller has a best response consisting in choosing a starting point in $Q$ and going directly to the attack point when she is informed of the attack.
Thus, with the same mappings of strategies in the hiding game $H'=(Q,m)$ to strategies in $P'$ as before, any quantity guaranteed by the searcher in $H'$ is guaranteed by the patroller in $P'$. Conversely, any quantity guaranteed by the hider in $H'$ is guaranteed by the attacker in $P'$. Thus the values of these two games are the same. 

\section{Patrolling games}
\label{sec_patrol}
\subsection{The value of patrolling games}
\label{subsec_valuepatrol}

Our first result is the existence of the value of patrolling games. We denote it $V_Q(m,r)$. In addition, we prove that the patroller has an optimal strategy and the attacker has an $\varepsilon$-optimal strategy with finite support. The fact that the patroller has an optimal strategy means that she can guarantee that the probability of detecting the attack is at least $V_Q(m,r)$, no matter what the attacker does. Similarly, the attacker can guarantee that the probability of being caught is at most $V_Q(m,r)$, up to $\varepsilon$, no matter what the patroller does. Hence, in patrolling games, the value represents the probability (up to $\varepsilon$) of the attack being intercepted when both the patroller and the attacker play ($\varepsilon$-)optimally.

\begin{proposition}
\label{prop_existence_value}
The patrolling game $(Q,m,r)$ played with mixed strategies has a value $V_Q(m,r)$.

Moreover the patroller has an optimal strategy and the attacker has an $\varepsilon$-optimal strategy with finite support, i.e.,
\begin{align*}
V_Q(m,r) &= \max_{\mu\in\Delta(\W)}\inf_{(y,t)\in \A} \int_\W g_{m,r}(w,(y,t))d\mu(w)\\
&= \inf_{\nu\in\Delta_f(\A)}\max_{w\in \W} \int_\A g_{m,r}(w,(y,t))d\nu(y,t).
\end{align*}
\end{proposition}
The proof of \cref{prop_existence_value} is provided in the \hyperlink{app:appendix_patrol}{Appendix of Chapter 4}.

\subsection{Decomposition}

As it is the case for graphs, see \citep{alpernetal2011}, it is possible to consider a search space $Q$ as the union of simpler search spaces $Q_1,\dots,Q_n$ for which the values of the corresponding patrolling games may be known. The value of the original patrolling game is lower bounded by a function of the values of the patrolling games involved in the decomposition.

\begin{proposition}
\label{prop_decomposition}
Let $Q$ and $Q_1,\dots,Q_n$ be search spaces such that $Q = \cup_{i=1}^n Q_i$. Then for all $m,r\in \R_+$
\[V_Q(m,r)\geq \frac{1}{\sum_{i=1}^n V_{Q_i}(m,r)^{-1}}.\]
\end{proposition}

\begin{proof}
For all $i\in\{1,\dots,n\}$, let $\mu_i$ be an optimal strategy of the patroller in the game $(Q_i,m,r)$. Let $\mu$ be the strategy of the patroller in the game $(Q,m,r)$ which consists in playing strategy $\mu_i$ with probability $\frac{V_{Q_i}(m,r)^{-1}}{\sum_{k=1}^n V_{Q_k}(m,r)^{-1}}$.

Let $(y,t)\in\A = Q\times \R_+$ be an attack in the game $(Q,m,r)$. Then there exists $i\in\{1,\dots,n\}$ such that $y\in Q_i$, hence $g(\sigma,(y,t))\geq \frac{V_{Q_i}(m,r)^{-1}}{\sum_{k=1}^n V_{Q_k}(m,r)^{-1}} V(Q_i,m,r).$
\end{proof}

\subsection{A general upper bound}

Our goal is now to obtain a general upper bound for the value of patrolling games. As in \citep{AlpGal:Kluwer2003}, let us introduce the maximal rate at which the patroller can discover new points of $Q$.

\begin{definition}
\label{def_maxdisco}
The maximal discovery rate is given by
\[\rho = \sup_{w\in\W, t>0}\frac{\lambda(w([0,t])+B_r(0))-\lambda(B_r)}{t},\]
where $w([0,t]) = \{w(\tau) \ | \ \tau\in [0,t]\}$ is the image of $[0,t]$ by $w$ and $w([0,t])+B_r(0) = \{y\in \R^n \ | \ d(w([0,t],y))\leq r\}$.
\end{definition}
Hence, in a network the maximal discovery rate $\rho$ is $1$, in $\R^2$ and $\R^3$ endowed with the Euclidean norm, under suitable assumptions, the maximal discovery rate is typically $2r$ and $\pi r^2$ respectively, that is the sweep width of the patroller.

Let us now give a general upper bound on the value of patrolling games whose search space have nonzero Lebesgue measure. It is the upper bound used to prove \cref{th_euleriannetwork,thm_simplesearchspace,theorem_equivalent_static}. A similar bound is given in \citep{alpernetal2011} in the discrete case.

\begin{proposition}
\label{proposition_upperbound}
Let $Q$ be a search space such that $\lambda(Q)>0$. Then \[V_Q(m,r) \leq \frac{m\rho+\lambda(B_r)}{\lambda(Q)}.\]
\end{proposition}

To prove \cref{proposition_upperbound}, we define a strategy for the attacker called the uniform strategy, under which he attacks uniformly over $Q$ at time $0$. Intuitively, a best response of the patroller is to cover as much points in $Q$ as possible between time $0$ and time $m$.

\begin{definition}
\label{def_attackunif}
Let $Q$ be a search space such that $\lambda(Q)>0$. The attacker's uniform strategy on $Q$, denoted $a_\lambda$, is a random choice of the attack point at time $0$ such that for all measurable sets $B\subset Q$,
\[a_\lambda(B,0) = \frac{\lambda(B)}{\lambda(Q)}, \text{ and } a_\lambda(B,t) = 0 \text{ if } t>0.\]
\end{definition}

\begin{proof}[Proof of \cref{proposition_upperbound}]
For all $w\in\W$, the payoff to the patroller when the attacker plays $a_\lambda$ is
\begin{align*}
\int_\A g_{m,r}(w,(y,t))da_\lambda(y,t) &= \frac{\lambda\big(w([0,m])+B_r(0)\big)}{\lambda(Q)}\\ &\leq \frac{m\rho + \lambda(B_r)}{\lambda(Q)}.
\end{align*}
\end{proof}

\subsection{Patrolling a network}
\label{subsec_network}


\subsubsection{Definition of a network}
\label{subsubsec_definitionnetwork}

We follow the construction of a network of \cite{fournier2016}. Let $(V,E,l)$ be a weighted undirected graph, $V$ is the finite set of nodes and $E$ the finite set of edges whose elements $e\in E$ have length $l(e)\in\R_+$. An edge $e\in E$ linking the two nodes $u$ and $v$ is also denoted $(u,v)$.

We identify the elements of $V$ with the vectors of the canonical basis of $\R^{|V|}$. The network generated by $(V,E)$ is the set of points
\[\n=\{(u,v,\alpha) \ | \ \alpha\in [0,1] \text{ and } (u,v)\in E\},\] where $(u,v,\alpha) = \alpha u + (1-\alpha)v$.

We denote $P(u,v)$ the set of paths between two points $u$ and $v$ in $\n$. It is the set of all sequences $(u_1,\dots, u_n)$, $n\in \N^\ast$ such that $u_1 = u$, $u_n = v$ and such that for all $i \in \{1,\dots,n-1\}$, $u_i$ and $u_{i+1}$ belong to the same edge. Let $u_1=(u,v,\alpha_1)$ and $u_2=(u,v,\alpha_2)$, and suppose $\alpha_1<\alpha_2$. The set $[u_1,u_2]=\{(u,v,\alpha)\ | \ \alpha\in [\alpha_1,\alpha_2]\}$ is called an interval.

Finally, networks are endowed with their natural metric $d$, i.e., the length of the shortest path between two points, as well as their Lebesgue measure defined as a natural extension of the Lebesgue measure on real intervals.

\subsubsection{Eulerian networks}
\label{subsubsec_eulerian}

For Eulerian networks, it is possible to compute the value and optimal strategies of the game. As stated in the next definition, an Eulerian tour is a closed path in $\n$ visiting all points and having length $\lambda(\n)$. 

\begin{definition}
Let $u\in\n$ and $\pi = (u_1,u_2,\dots,u_{n-1},u_n)\in P(u,u)$. If \[\bigcup_{k=1}^{n-1}[u_k,u_{k+1}] = \n,\] then $\pi$ is called a tour. Moreover, if $\sum_{k=1}^{n-1}\lambda([u_k,u_{k+1}]) = \lambda(\n),$ then $\pi$ is called an Eulerian tour. A network $\n$ is said to be Eulerian if it admits an Eulerian tour.
\end{definition}

\begin{example}
\cref{network_1,network_2} display two examples of networks. $\n_1$ is an Eulerian network with Eulerian tour $\pi_1 = (u_1,u_2,u_3,u_4,u_5,u_6,u_3,u_7,u_1)$. In contrast, $\n_2$ is not an Eulerian network.
\begin{figure}[!ht]
\centering
  \begin{minipage}[b]{0.4\textwidth}
    \begin{tikzpicture}[scale=0.55,rotate=90]
\draw[-] (0,2) to (-2,0);
\fill (0,2) circle (3pt) node[above] {$u_1$};
\fill (-2,0) circle (3pt) node[left] {$u_2$};
\draw[-] (0,2) to (2,0);
\fill (2,0) circle (3pt) node[right] {$u_7$};
\draw[-] (-2,0) to (2,-4);
\fill (2,-4) circle (3pt) node[right] {$u_4$};
\draw[-] (2,0) to (-2,-4);
\fill (-2,-4) circle (3pt) node[left] {$u_6$};
\draw[-] (2,-4) to (0,-6);
\fill (0,-6) circle (3pt) node[below] {$u_5$};
\draw[-] (-2,-4) to (0,-6);
\fill (0,-2) circle (3pt) node[right] {$u_3$};
\end{tikzpicture}
    \caption{The network $\n_1$ is Eulerian}
    \label{network_1}
  \end{minipage}
  \hfill
  \begin{minipage}[b]{0.4\textwidth}
    \begin{tikzpicture}[scale=0.55,rotate=90]
\draw[-] (0,4) to (-2,0);
\draw[-] (0,4) to (2,0);
\draw[-] (-2,0) to (0,-4);
\draw[-] (2,0) to (0,-4);
\draw[-] (0,4) to (0,-4);
\fill (0,4) circle (3pt) node[above] {$u_1$};
\fill (-2,0) circle (3pt) node[left] {$u_2$};
\fill (2,0) circle (3pt) node[right] {$u_4$};
\fill (0,-4) circle (3pt) node[below] {$u_3$};
\end{tikzpicture}
    \caption{The network $\n_2$ is not Eulerian}
    \label{network_2}
  \end{minipage}
\end{figure}
\end{example}

Our objective is now to define the uniform strategy of the patroller for Eulerian networks. This strategy is optimal for Eulerian networks. First, we need to define a parametrization of the network.

\begin{definition}
\label{def_param_network}
Let $\n$ be an Eulerian network. A continuous function $w$ from $[0,\lambda(\n)]$ to $\n$ is called a parametrization of $\n$ if it satisfies
\begin{itemize}
\item[i)] $w(0) = w(\lambda(\n))$;
\item[ii)] $w$ is surjective;
\item[iii)] $\forall t_1,t_2 \in [0,\lambda(\n)]$ $\lambda(w([t_1,t_2])) = |t_1-t_2|$ (the speed of $w$ is $1$).
\end{itemize}
Moreover such a function $w$ can be extended to a $\lambda(\n)$-periodic function on $\R_+$ which is still denoted $w$.
\end{definition}

\begin{lemma}
\label{lem_tour}
Let $\n$ be an Eulerian network, then there exists a parametrization of $\n$.
\end{lemma}

The proof of \cref{lem_tour} is provided in the \hyperlink{app:appendix_patrol}{Appendix of Chapter 4}. It is now possible to define the uniform strategy of the patroller. The idea behind this strategy is that the patroller uniformly chooses a starting point in $\n$, and then follows a parametrization as in \cref{def_param_network} above.

\begin{definition}
Suppose $\n$ is an Eulerian network. Let $w$ be a parametrization of $\n$. Denote $(w_{t_0})_{t_0\in[0,\lambda(\n)]}$ the family of $\lambda(\n)$-periodic walks such that $w_{t_0}(\cdot) = w(t_0 + \cdot).$
The patroller's uniform strategy is given by the uniform choice of $t_0\in [0,\lambda(\n)]$.
\end{definition}

The next theorem is the main result on patrolling games for networks. It gives a simple expression of the value of patrolling games played on any Eulerian network. The result relies on the fact that for such networks, the patroller can achieve the upper bound of \cref{proposition_upperbound} using her uniform strategy. Note that \cref{th_euleriannetwork} below is related to \citep[Theorem 1]{alpernetal2016a}, as well as \citep[theorem 13]{alpernetal2011} for Hamiltonian graphs in the discrete case.

\begin{theorem}
\label{th_euleriannetwork}
If $\n$ is an Eulerian network, then \[V_{\n}(m,0)=\min\left(\frac{m}{\lambda(\n)},1\right).\] Moreover the attacker's and the patroller's uniform strategies are optimal.
\end{theorem}
The proof of \cref{th_euleriannetwork} is provided in the \hyperlink{app:appendix_patrol}{Appendix of Chapter 4}. It is interesting to note that in search games with an immobile hider, the uniform strategies of the searcher and the hider are also optimal in Eulerian networks.

\subsubsection{The three parallel arc network}

In this section, we compute bounds on the value of a patrolling game played over the network with three parallel arcs, which plays an important role in the search game literature. Indeed despite its simple shape, finding the value and optimal strategies is a difficult task, as mentioned in the introduction.

We consider again the network $\n_2$ represented in \cref{network_2}. In this example, we take $l(u_1,u_2)=l(u_2,u_3)=l(u_1,u_4)=l(u_4,u_3)=1/2$ and $l(u_1,u_3)=1$. Notice that $\lambda(\n_2)=3$. We compute the following bounds on the value of $(\n_2,m,0)$:

\[V_{\n_2}(m,0)
\left\{
    \begin{array}{ll}
= \frac{m}{3} &\text{ if } m\leq 2\\
\in \left[ \frac{5m-2}{3(m+2)} , 1-\frac{1}{3}\left(\frac{4-m}{2}\right)^2 \right] &\text{ if } m\in \left[2,\frac{10}{3}\right]\\
\in \left[ \frac{14-2m}{3(6-m)} , 1-\frac{1}{3}\left(\frac{4-m}{2}\right)^2 \right]&\text{ if } m\in \left[\frac{10}{3 },4\right]\\
= 1 \text{ if }& m\geq 4.
    \end{array}
\right.\]
These bounds are plotted on \cref{fig_bounds} below.
\begin{figure}[!ht]
\centering
\begin{tikzpicture}[xscale=1,yscale=2]
\draw[->] (0,0) -- (6,0);
\draw (6,0) node [above] {$m$};
\draw[->] (0,0) -- (0,1.5);
\draw (0,1.5) node [right] {$V_{\n_2}(m,0)$};
\draw (2,0) node [below] {$2$};
\draw (4,0) node [below] {$4$};
\draw (0,1) node [left] {$1$};
\draw (0,2/3) node [left] {$2/3$};
\draw [dashed] (2,0) -- (2,2/3);
\draw [dashed] (4,0) -- (4,1);
\draw [dashed] (0,2/3) -- (2,2/3);
\draw [dashed] (0,1) -- (4,1);
\draw[domain=0:2] plot(\x,{\x/3});
\draw[domain=2:10/3] plot(\x,{(5*\x-2)/(3*(\x+2))});
\draw[domain=10/3:4] plot(\x,{(14-2*\x)/(3*(6-\x))});
\draw[domain=2:4] plot(\x,{1-1/3*((4-\x)/2)^2});
\draw[domain=4:6] plot(\x,{1});
\end{tikzpicture}
\caption{Bounds on the value of the game as a function of $m$}
\label{fig_bounds}
\end{figure}
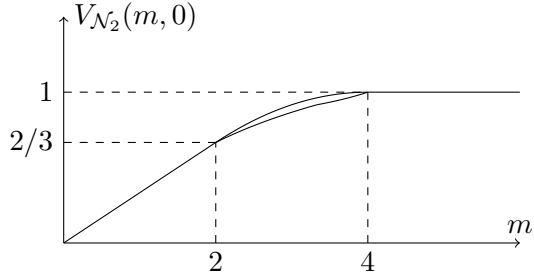

\paragraph{First case: $m\leq 2$.} Recall that (\cref{proposition_upperbound}) $V_{\n_2}(m,0) \leq \frac{m}{3}$ for all $m\geq 0$. Suppose the patroller uniformly chooses one of the three Eulerian sub-networks of $\n_2$ of length $2$. Then by \cref{th_euleriannetwork} she guarantees $m/2$ in these sub-networks. By symmetry, the patroller guarantees $\frac{2}{3}\frac{m}{2} = \frac{m}{3}$ in $\n_2$.
We also give an alternative strategy of the patroller which guarantees $m/3$, and which will be useful for the next case $2<m<4$.
Let $\pi^1 = (u_1,u_2,u_3,u_1,u_4,u_3)$ and $\pi^2 = (u_3,u_2,u_1,u_3,u_4,u_1)$ be two paths. $\pi^1$ and $\pi^2$ naturally induce two walks on $[0,3]$ at speed $1$, respectively denoted $w^1$ and $w^2$. For all $u\in \n_2$ and all $i\in\{1,2\}$ there exists a unique $t^i_u\in [0,3]$ such that $w^i(t^i_u) = u$, except for $u_1$ and $u_3$ for which the corresponding $t^i_u$ may be taken arbitrarily. Now for all $u\in \n_2$ and all $t\in\R_+$, define
\[w^1_u(t)=\left\{
    \begin{array}{ll}
w^1(t+t^1_u) &\text{ if } t\in \left[0,3-t^1_u\right]\\
w^2\left(t-(3(2k+1)-t^1_u)\right) &\text{ if } t\in \left(3(2k+1)-t^1_u,3(2k+2)-t^1_u\right]\\
w^1\left(t-(3(2k+2)-t^1_u)\right) &\text{ if } t\in \left(3(2k+2)-t^1_u,3(2k+3)-t^1_u\right]
    \end{array}
\right.\]
for all $k\in\N$. The walk $w^1_u$ starts at $t^1_u$ and alternates between following $w^1$ and $w^2$. The walk $w^2_u$ is defined analogously by switching the superscripts $1$ and $2$ in the definition above. Denote $\mu^0$ the uniform choice of a walk in $(w^i_u)^{i\in \{1,2\}}_{u\in \n_2}$.
It is not difficult to check that $\mu^0$ guarantees $m/3$ to the patroller (moreover $\mu^0$ yields a payoff of $m/3$ for every $(y,t)\in \A$). Hence $V_{\n_2}(m,0) = \frac{m}{3}.$

\paragraph{Second case: $2< m <4 $.} We detail the computation for $m=3$. The walks $w^3$, $w^4$ and $w^5$ hereafter can be adapted and similar strategies can be used to derive the bounds for all $m\in(2,4)$.

Let us define three paths $\pi^3$, $\pi^4$ and $\pi^5$ as in \cref{fig_w1,fig_w2,fig_w3} respectively. That is, $\pi^3 = (u_1,u_2,u_3,u_5,u_3,u_1,u_6,u_1)$, $\pi^4 = (u_1,u_7,u_1,u_3,u_8,u_3,u_4,u_1)$ and $\pi^5 = (u_1,u_2,u_3,u_{10},u_3,u_4,u_1,u_9,u_1)$. Where $u_5 = (u_3,u_4,1/2)$, $u_6 = (u_1,u_4,1/2)$, $u_7 = (u_1,u_2,1/2)$, $u_8 = (u_2,u_3,1/2)$, $u_9 = (u_1,u_3,1/4)$ and $u_{10} = (u_1,u_3,3/4)$.

\begin{figure}[!ht]
  \begin{minipage}[b]{0.3\textwidth}
    \begin{tikzpicture}[scale=0.5]
\draw[-] (0,4) to (-2,0);
\draw[-] (0,4) to (2,0);
\draw[-] (-2,0) to (0,-4);
\draw[-] (2,0) to (0,-4);
\draw[-] (0,4) to (0,-4);

\draw[dashed] (1.2,2.2) to (0,4.4);
\draw[dashed] (0,4.4) to (-2.25,0);
\draw[dashed] (-2.25,0) to (0,-4.4);
\draw[dashed] (0,-4.4) to (1.2,-2.2);
\draw[dashed] (1.2,-2.2) to (0.8,-1.8);
\draw[dashed] (0.8,-1.8) to (0.2,-3.0);
\draw[dashed] (0.2,-3.0) to (0.2,3);
\draw[dashed] (0.2,3) to (0.8,1.8);
\draw[dashed] (0.8,1.8) to (1.2,2.2);

\fill (0.2,4.4) node[above] {$u_1$};
\fill (-2.2,0) node[left] {$u_2$};
\fill (2.2,0) node[right] {$u_4$};
\fill (0.2,-4.4) node[below] {$u_3$};
\fill (1,-2.8) node[right] {$u_5$};
\fill (1,2.8) node[right] {$u_6$};
\end{tikzpicture}
    \caption{The path $\pi^3$}
    \label{fig_w1}
  \end{minipage}
    \hfill
  \begin{minipage}[b]{0.3\textwidth}
    \begin{tikzpicture}[scale=-0.5]
\draw[-] (0,4) to (-2,0);
\draw[-] (0,4) to (2,0);
\draw[-] (-2,0) to (0,-4);
\draw[-] (2,0) to (0,-4);
\draw[-] (0,4) to (0,-4);
\draw[dashed] (1.2,2.2) to (0,4.4);
\draw[dashed] (0,4.4) to (-2.25,0);
\draw[dashed] (-2.25,0) to (0,-4.4);
\draw[dashed] (0,-4.4) to (1.2,-2.2);
\draw[dashed] (1.2,-2.2) to (0.8,-1.8);
\draw[dashed] (0.8,-1.8) to (0.2,-3.0);
\draw[dashed] (0.2,-3.0) to (0.2,3);
\draw[dashed] (0.2,3) to (0.8,1.8);
\draw[dashed] (0.8,1.8) to (1.2,2.2);

\fill (1,-2.8) node[left] {$u_7$};
\fill (1,2.8) node[left] {$u_8$};

\fill (-0.2,-4.4) node[above] {$u_1$};
\fill (2.2,0) node[left] {$u_2$};
\fill (-2.2,0) node[right] {$u_4$};
\fill (-0.2,4.4) node[below] {$u_3$};
\end{tikzpicture}
    \caption{The path $\pi^4$}
    \label{fig_w2}
  \end{minipage}  \hfill
  \begin{minipage}[b]{0.3\textwidth}
    \begin{tikzpicture}[scale=0.5]
\draw[-] (0,4) to (-2,0);
\draw[-] (0,4) to (2,0);
\draw[-] (-2,0) to (0,-4);
\draw[-] (2,0) to (0,-4);
\draw[-] (0,4) to (0,-4);
\draw[dashed] (1.6,0) to (0.2,-3.0);
\draw[dashed] (1.6,0) to (0.2,3);
\draw[dashed] (-1.6,0) to (-0.2,-3.0);
\draw[dashed] (-1.6,0) to (-0.2,3);
\draw[dashed] (0.2,-3.0) to (0.2,-1.5);
\draw[dashed] (0.2,3.0) to (0.2,1.5);
\draw[dashed] (-0.2,-3.0) to (-0.2,-1.5);
\draw[dashed] (-0.2,3.0) to (-0.2,1.5);
\draw[dashed] (0.2,1.5) to (-0.2,1.5);
\draw[dashed] (0.2,-1.5) to (-0.2,-1.5);

\fill (0.2,4.4) node[above] {$u_1$};
\fill (-2.2,0) node[left] {$u_2$};
\fill (2.2,0) node[right] {$u_4$};
\fill (0.2,-4.4) node[below] {$u_3$};

\fill (0,1.3) node[right] {$u_9$};
\fill (0,-1.3) node[right] {$u_{10}$};

\end{tikzpicture}
    \caption{The path $\pi^5$}
    \label{fig_w3}
  \end{minipage}
\end{figure}

$\pi^3$, $\pi^4$ and $\pi^5$ naturally induce three $3$-periodic walks at speed $1$, denoted respectively $w^3$, $w^4$ and $w^5$. These are such that for $i\in\{3,4,5\}$, $w^i$ intercepts any attack on $w^i([0,3])$ with probability $1$.

With a slight abuse of notation, for $y\in [0,1/2]$, denote $y$ the point $(u_1,u_3,y)$ and $1-y$ the point $(u_1,u_3,1-y)$. By symmetry it is enough to consider attacks occurring at $y$. Moreover, $\mu^0$, $w^3$, $w^4$ and $w^5$ make the patroller time indifferent, hence we only consider attacks at time $0$.

$\mu^0$ intercepts the attack $(y,0)$ with probability $1-\frac{1-2y}{6}=\frac{5}{6}+\frac{y}{3}$. Indeed, only the walks $w^1_u$, such that $u$ belongs to the open interval $\{(u_1,u_3,\alpha)\ \vert \ \alpha\in (y,1-y)\}$ do not intercept the attack.
Finally, define $\widetilde{\mu} = \frac{1}{15}(\delta_{w^3}+\delta_{w^4}+\delta_{w^5})+\frac{4}{5}\mu^0$, where $\delta_w$ is the Dirac measure at $w\in \W$. 

At any time, an attack at $y\leq \frac{1}{4}$ is intercepted by $\widetilde{\mu}$ with probability
\[\frac{1}{15}\cdot 3+\frac{4}{5}\left(\frac{5}{6}+\frac{y}{3}\right)\geq\frac{3}{15}+\frac{4}{5}\cdot\frac{5}{6} = \frac{13}{15}.\]
An attack at $y> 1/4$ is intercepted by $\widetilde{\mu}$ with probability
\[\frac{1}{15}\cdot 2 + \frac{4}{5}\left(\frac{5}{6}+\frac{y}{3}\right) \geq\frac{2}{15} + \frac{4}{5}\left(\frac{5}{6}+\frac{1}{12}\right) = \frac{13}{15}.\]
Hence $V_{\n_2}(3,0)\geq \frac{13}{15}.$

Define the following attack $\widetilde{a}$: choose uniformly a point in $\n_2\times [0,3]$. The tour $(u_1,u_2,u_3,u_1,u_4,u_3,u_2,u_1,u_3,u_4,u_1)$ induces a $6$-periodic walk $w^6$ which is a best response for the patroller. Moreover $g_{3,0}(w^6,\widetilde{a}) = 11/12$.
Hence $V_{\n_2}(3,0)\leq \frac{11}{12}.$

\paragraph{Third case: $m\geq 4$.}
The tour $(u_1,u_2,u_3,u_1,u_4,u_3,u_1)$ induces a $4$-periodic walk which guarantees $1$ to the patroller. Hence $V_{\n_2}(m,0) = 1.$

\subsection{Patrolling a simple search space in \texorpdfstring{$\R^2$}{R2}}
\label{subsec_simplesearch}

In this section, we are interested in patrolling games in $\R^2$ for a large class of search spaces called simple search spaces. To introduce this class of search spaces, we first need to recall the notion of total variation of a function $f$.

\begin{definition}
Let $a>0$. Let $f : [0,a]\to \R^n$ be a continuous function. Then the total variation of $f$ is the quantity:
\[TV(f) = \sup \left\lbrace\sum_{i=1}^n \eucl{f(t_i) - f(t_{i-1})} \biggm| n\in\N^\ast, 0=t_0<t_1<\dots<t_n=a\right\rbrace.\]
If $TV(f)<+\infty$, then $f$ is said to have bounded variation.
\end{definition}

The next definition introduces a classical assumption on the boundary of a search space in $\R^2$. This is a weak assumption already made in \citep{Gal:AP1980} and \citep{AlpGal:Kluwer2003}.

\begin{definition}
Let $a>0$, let $f_1$ and $f_2$ be two continuous functions from $[0,a]$ to $\R$ such that $f_1\geq f_2$, $f_1\neq f_2$, and $f_1$ and $f_2$ have bounded variation.
Then the nonempty compact set $\{(x,t)\in[0,a]\times \R \ | \ f_2(x)\leq t \leq f_1(x)\}$ is called an elementary search space.

Let $Q$ be the finite union of elementary search spaces such that any two have disjoint interiors. If $Q$ is path-connected, then it is called a simple search space.
\end{definition}

The next theorem is the main result on patrolling games on simple search spaces. It gives a simple asymptotic expression of the value as the detection radius $r$ goes to $0$. The result relies on the fact that the patroller can use a uniform strategy in the spirit of what has been done in the previous section for Eulerian networks. This strategy yields a lower bound that asymptotically matches the upper bound of \cref{proposition_upperbound}.

As one would expect, the value goes to $0$ as $r$ goes to $0$. It is interesting to note that due to the movement of the patroller the convergence is linear in $r$ and not quadratic. Indeed, the relevant parameter is the sweep width of the patroller and not the area of detection.

\begin{theorem}
\label{thm_simplesearchspace}
If $Q$ is a simple search space endowed with the Euclidean norm, then \[V_Q(m,r)\sim \frac{2rm}{\lambda(Q)},\] as $r$ goes to $0$.
\end{theorem}

The proof of \cref{thm_simplesearchspace} is provided in the \hyperlink{app:appendix_patrol}{Appendix of Chapter 4}.

\section{Hiding games}
\label{sec_hiding}

Recall that the value of a hiding game is equal to the value of a patrolling game with time limit $m$ equal to $0$. Hiding games have a value which represents the probability (up to $\varepsilon$) that the searcher and the hider are at distance less that $r$ when they play ($\varepsilon$-)optimally.

\begin{proposition}
\label{prop_existence_value_hide}
The hiding game $(Q,r)$ played in mixed strategies has a value denoted $V_Q(r)$. Moreover the searcher has an optimal strategy and the hider has an $\varepsilon$-optimal strategy with finite support.
\end{proposition}

\subsection{Equalizing strategies}
\label{subsec_equilizing}

We now study particular strategies called equalizing, these have been introduced by \cite[definition 7.3 and proposition 7.3]{bishopevans2013} when the search space is a graph. We adapt those considerations to our compact setting. A similar notion also appears in \citep[lemma 2.6]{lidbetter2013}.

In hiding games, both players have the same strategy sets and the payoff function is symmetric in the sens that, if $\mu\in\Delta(Q)$ and $y\in Q$, then $h_r(\mu,y) = h_r(y,\mu) = \mu(B_r(y)\cap Q)$.

A strategy for one player is equalizing if the payoff, when this player uses this strategy, does not depend on the strategy of the other player. The interest of equalizing strategies lies in the fact that if such a strategy exists, then it is optimal for both players. 
\begin{definition}Let $Q$ be a search space.
A strategy $\mu\in\Delta(Q)$ is said to be equalizing if there exists $c\in [0,1]$ such that $h_r(\mu,y) = c$ for all $y\in Q$. 
\end{definition}

\begin{proposition}
\label{prop_equalizing}
Let $\mu\in\Delta(Q)$. Then $\mu$ is an equalizing strategy (with constant payoff $c$) if and only if $\mu$ is optimal for both players (and in that case $V_Q(r)=c$).
\end{proposition}

\begin{proof}
Suppose $\mu\in\Delta(Q)$ is an equalizing strategy. If the searcher plays $\mu$, then for all $y\in Q$ $\mu(B_r(y)\cap Q) = c$, hence $V_Q(r)\geq c$. Symmetrically, if the hider plays $\mu$, then for all $x\in Q$ $\mu(B_r(x)\cap Q) = c$, hence $V_Q(r)\leq c$, and $V_Q(r) = c$.

Conversely, suppose $\mu\in\Delta(Q)$ is optimal for both players. Then the searcher guaranties $V_Q(r)$ that is for all $y\in Q$ $\mu(B_r(y)\cap Q) \geq V_Q(r)$, and the hider guaranties $V_Q(r)$ that is for all $x\in Q$ $\mu(B_r(x)\cap Q) \leq V_Q(r)$. Hence for all $y\in Q$ \[\mu(B_r(y)\cap Q) = V_Q(r).\]
\end{proof}

The following game is an example of a hiding game with finite search space without equalizing strategies.

\begin{example}
Let $r=1$ and $Q = \{x_1,x_2,x_3,x_4,x_5\}$ be the finite subset of $\R^2$ such that $x_1 = (0,0)$, $x_2= (0,1)$, $x_3 = (1,1)$, $x_4 = (1,0)$ and $x_5 = (1/2,0)$.
Denote for $i\in \{1,\dots,5\}$ $Q_i = \{j\in\{1,\dots,5\} \ \vert \ \eucl{x_i-x_j} \leq r\}$.
That is $Q_1=\{1,2,4,5\}$, $Q_2=\{1,2,3\}$, $Q_3=\{2,3,4\}$, $Q_4=\{1,3,4,5\}$ and $Q_5=\{1,4,5\}$.

The game $(Q,r)$ admits an equalizing strategy if and only if the following system of equations admits a solution $p=(p_i)_{1\leq i \leq 5}$:
\begin{align}
\label{eq_finite}
\left\{
    \begin{array}{l}
p_i\geq 0 \text{ for all } i\in\{1,\dots,5\}\\
\sum_{i=1}^5 p_i = 1\\
\sum_{i\in Q_1} p_i = \sum_{i\in Q_j} p_i \text{ for all } j\in\{2,\dots,5\}. 
    \end{array}
\right.
\end{align}
It is easy to verify that this system does not admit a solution, hence the game $(Q,r)$ does not have an equalizing strategy.
\end{example}

\subsection{An asymptotic result for hiding games}
\label{subsec_asymptotichinding}

The next theorem is the main result on hiding games. For any search space $Q\subset \R^n$ with positive Lebesgue measure, it gives a simple asymptotic expression of the value when the detection radius goes to $0$. In this static setting the value is equivalent, as $r$ goes to $0$, to the ratio of the volume of the ball of radius $r$ over the volume of $Q$. This result relies on the fact that the searcher has a strategy that yields a lower bound which asymptotically matches the upper bound of \cref{proposition_upperbound}.

\begin{theorem}
\label{theorem_equivalent_static}
Let $Q$ be a compact subset of $\R^n$. Suppose $\lambda(Q)>0$. Then \[V_Q(r) \sim \frac{\lambda(B_r)}{\lambda(Q)}\] as $r$ goes to $0$.
\end{theorem}

The proof of \cref{theorem_equivalent_static} is provided in the \hyperlink{app:appendix_patrol}{Appendix of Chapter 4}. A consequence of \cref{theorem_equivalent_static} is that for a compact set $Q$ included in $\R^n$ such that $\lambda(Q)>0$, $V_Q(r)\sim r^n \frac{\lambda(B_1)}{\lambda(Q)}$ as $r$ goes to $0$. When $\lambda(Q) = 0$, it is not always the case that $V_Q$ admits an equivalent of the form $Mr^\alpha$, with $\alpha$ and $M$ positive, as $r$ goes to $0$, as it is shown in \cref{example_cantor}.

\begin{example}
\label{example_cantor}
Let $Q\subset [0,1]$ be the following Cantor-type set.
Define $C_0 = [0,1]$, and for all $n\in \N^{\ast}$ $C_n = \frac{1}{4}C_{n-1} \cup \left(\frac{3}{4}+\frac{1}{4}C_{n-1}\right)$. Finally, let $Q = \bigcap\limits_{n\in \N} C_n$. $Q$ is compact and $\lambda(Q) = 0$.

The value of the hiding game played on $Q$ is given by the following formula:
\[V_Q(r)=\left\{
    \begin{array}{ll}
\frac{1}{2^n} &\text{ if } r \in \left[\frac{1}{2^{2n}},\frac{3}{2^{2n}}\right),\\
\frac{1}{2^{n-1}} &\text{ if } r \in \left[\frac{3}{2^{2n}},\frac{1}{2^{2(n-1)}}\right), n\in \N^\ast.
    \end{array}
\right.
\]
Indeed, let $\Sigma_1 = \{0,1\}$ and for all $n\in\N^\ast\setminus\{1\}$ let $\Sigma_n = \frac{1}{4}\Sigma_{n-1} \cup \left(\frac{3}{4}+\frac{1}{4}\Sigma_{n-1}\right)$.
For $n\in\N^{\ast}$, consider the following strategy $\sigma_n$: choose uniformly a point in $\Sigma_n$, that is with probability $\frac{1}{|\Sigma_n|} = \frac{1}{2^n}$.
Let $n\in\N^\ast$ suppose $r\in\left[\frac{1}{2^{2n}},\frac{3}{2^{2n}}\right)$. Then for all $q\in Q$ there is exactly one point $s$ in $\Sigma_n$ such that $|q-s|\leq r$. Hence $\sigma_n$ is an equalizing strategy which guarantees $\frac{1}{2^n}$ to both players.

Let $\Sigma'_1=\{\frac{1}{4}\}$ and for all $n\in\N^\ast\setminus\{1\}$ let $\Sigma'_n = \frac{1}{4}\Sigma'_{n-1} \cup \left(1-\frac{1}{4}\Sigma'_{n-1}\right)$. For $n\in\N^\ast$ consider the following strategy $\sigma'_n$: choose uniformly a point in $\Sigma'_n$, that is with probability $\frac{1}{|\Sigma'_n|} = \frac{1}{2^{n-1}}$.
Suppose now that $r \in \left[\frac{3}{2^{2n}},\frac{1}{2^{2(n-1)}}\right)$. Then for all $q\in Q$ there is exactly one point $s$ in $\Sigma'_n$ such that $|q-s|\leq r$. Hence $\sigma'_n$ is an equalizing strategy which guarantees $\frac{1}{2^{n-1}}$ to both players.

In particular, for all $n\in\N^\ast$ 
\[V_Q\left(\frac{1}{2^{2n-1}}\right)= V_Q\left(\frac{1}{2^{2n}}\right) = \frac{1}{2^n}.\]
Let $(r_n)_{n \in \N^\ast } = \left(\frac{1}{2^n}\right)_{n\in\N^\ast}$ and let $\alpha>0$. Then for all $n\in\N^\ast$
\begin{align*}
\frac{V_Q(r_{2n-1})}{(r_{2n-1})^\alpha} &= \frac{1}{2^\alpha}2^{(2\alpha-1)n} \text{ and } \frac{V_Q(r_{2n})}{(r_{2n})^\alpha} =2^{(2\alpha-1)n}.
\end{align*}
Thus we have
\[\lim_{n\to+\infty} \frac{V_Q(r_{2n-1})}{(r_{2n-1})^\alpha} = \left\{
    \begin{array}{ll}
+\infty &\text{ if } \alpha>1/2 \\
\frac{1}{\sqrt{2}} &\text{ if } \alpha = 1/2\\
0 &\text{ if } \alpha<1/2
    \end{array}
\right.
\text{ and } 
\lim\limits_{n\to+\infty} \frac{V_Q(r_{2n})}{(r_{2n})^\alpha} = 
\left\{
    \begin{array}{ll}
+\infty &\text{ if } \alpha>1/2 \\
1 &\text{ if } \alpha = 1/2\\
0 &\text{ if } \alpha<1/2.
    \end{array}
\right.
\]
Hence $r\mapsto V_Q(r)$ does not admit an equivalent of the form $r\mapsto M r^\alpha$, with $\alpha$ and $M$ positive numbers, as $r$ goes to $0$.
\end{example}

\section[Properties of the value function]{Properties of the value function of patrolling and hiding games}
\label{sec_propertiesvalue}

In this section we give some elementary properties of the function $V_Q$ for patrolling and hiding games.

\subsection{The value function of patrolling games}
\label{subsec_valuefuncpatrol}

\begin{proposition}
\label{prop_V}
Let $Q$ be a search space. The function 
\[\begin{array}{ccccc}
V_Q & : & \R_+\times \R_+ & \to & [0,1] \\
 & & (m,r) & \mapsto & V_Q(m,r) \\
\end{array}\] is
\begin{itemize}
\item[i)] non decreasing in $m$ and $r$;
\item[ii)] upper semi-continuous in $r$ for all $m$;
\item[iii)] upper semi-continuous in $m$ for all $r$.
\end{itemize}
\end{proposition}

\begin{proof}
Since i) is direct we only prove ii).
For all $(m,r)\in\R^2_+$,
\begin{align*}
V_Q(m,r) &= \max_{\mu\in \Delta(\W)} \inf_{(y,t)\in \A} \int_\W g_{m,r}(w,(y,t)) d\mu(w)\\
&= \inf_{(y,t)\in \A} \int_\W g_{m,r}(w,(y,t)) d\mu^\ast(w),
\end{align*}
where $\mu^\ast\in\Delta(\W)$ is an optimal strategy of the patroller.
Let $(y,t)\in \A$ and $m\geq 0$. For all $w\in \W$, the function $r\mapsto g_{m,r}(w,(y,t))$ is upper semi-continuous, as the indicator function of a closed set. Let $r_n\to r$, then by Fatou's lemma,
\begin{align*}
\limsup_n \int_W g_{m,r_n}(w,(y,t)) d\mu^\ast(w) &\leq \int_W \limsup_n g_{m,r_n}(w,(y,t)) d\mu^\ast(w)\\
&\leq \int_\W g_{m,r}(w,(y,t)) d\mu^\ast(w).
\end{align*}
Thus the function $r\mapsto \int_\W g_{m,r}(w,(y,t)) d\mu^\ast(w)$ is upper semi-continuous.
Hence \[V_Q(m,\cdot) : r\mapsto \inf_{(y,t)\in \A} \int_\W g_{m,r}(w,(y,t)) d\mu^\ast(w)\] is upper semi-continuous.

Since for all $w\in \W$, the function $m\mapsto g_{m,r}(w,(y,t))$ is upper semi-continuous, as the indicator function of a closed set, the proof of iii) is strictly analogous.
\end{proof}

\cref{example_unitinterval} in the next section shows that in general, for fixed $m$, $V_Q(\cdot,m)$ is not lower semi-continuous.

\begin{remark}
Let $m,r \geq 0$, and $Q_1, Q_2$ be two search spaces, it is clear that if $Q_1\subset Q_2$ then the attacker is better off in $Q_2$ hence $V_{Q_1}(m,r) \geq V_{Q_2}(m,r)$.
\end{remark}

\subsection{The value function of hiding games}
\label{subsec_valuefunchid}

Recall that the value of a hiding game is equal to the value of a patrolling game with time limit $m$ equal to $0$. Hence, the negative results presented in this section also hold for patrolling games when $m=0$.

The following simple example of a hiding game on the unit interval was first solved by \cite{ruckle1983}. It shows that in general, $V_Q$ is not lower semi-continuous.

\begin{example}
\label{example_unitinterval}
Let $Q$ be the $[0,1]$ interval, then 
\[V_Q(r) = 
\sigma(s,i) = \left\{
    \begin{array}{ll}
\min\left(\lceil\frac{1}{2r}\rceil^{-1},1\right) &\text{ if } r>0\\
0 &\text{ otherwise}.
    \end{array}
\right.
\]
Indeed, it is clear when $r$ equals $0$ and $r\geq 1/2$. Let $n\in \N^\ast$ and suppose $r \in \left[\frac{1}{2(n+1)},\frac{1}{2n}\right)$. Then the patroller guarantees $\frac{1}{n+1}$ by choosing equiprobably a point in $\left\lbrace\frac{1+2k}{2(n+1)}\right\rbrace_{0\leq k\leq n}$. And the attacker, choosing equiprobably a point in $\left\lbrace \frac{(2+\varepsilon) k }{2(n+1)} \right\rbrace_{0\leq k \leq n}$, with $0< \varepsilon \leq 2/n$, also guarantees $\frac{1}{n+1}$.
\end{example}

The next proposition disproves the somehow intuitive belief that the value of hiding games is continuous with respect to the Haussdorff metric between nonempty compact sets.

\begin{proposition}
\label{prop_V_Q}
Let $r\geq 0$. The function which maps any search space $Q$ to $V_Q(r)$ is in general not continuous with respect to the Hausdorff metric between nonempty compact sets.
\end{proposition}

\begin{proof}
Let $D_s = \{x\in\R^2 \ | \ \eucl{x}\leq s\}$ be the Euclidean disc of radius $s>0$ centered at $0$. From \cite{danskin1990}, it is known that
\[V_{D_s}(1) =\left\{
    \begin{array}{ll}
       1 &\text{ if } s \in [0,1]\\
\frac{1}{\pi}\arcsin\left(\frac{1}{s}\right) &\text{ if } s\in\left(1,\sqrt{2}\right].
    \end{array}
\right.\]
Hence $\lim_{s \to 1, s>1} V_{D_s}(1) = \frac{1}{2} < 1.$

The intuition is the following: it is clear that when $s$ equals $1$ the searcher guarantees $1$ by playing $x=(0,0)$. Suppose now that $s$ equals $1+\varepsilon$. Then the searcher covers almost all the area of the disc but less than half of its circumference. Hence the hider guarantees $1/2$ by choosing uniformly a point on the boundary of $D_s$.
\end{proof}

\section*{Appendix: omitted proofs}
\sectionmark{Appendix of Chapter 4}
\subsection*{Omitted proofs of \cref{subsec_valuepatrol}}

Let us first define a metric $D$ on the set $\W$ inducing the topology of compact convergence. For $n\in \N$, define $K_n=[0,n]$. Then 
\[\begin{array}{ccccc}
D & : & \W\times \W & \to & \R_+ \\
 & & (f,g) & \mapsto & \sum\limits_{n=1}^\infty \frac{1}{2^n}\sup\limits_{x\in K_n} \Vert f(x)-g(x)\Vert,
\end{array}\]
is a metric on $\W$, which induces the topology of compact convergence.

We recall the following fact about the topology of compact convergence. 

\begin{proposition}[Application of Theorem 46.2 in \citep{munkres2000}]
Let $Q$ be a search space. A sequence $f_n : \R_+ \to Q$ of functions converges to the function $f$ in the topology of compact convergence if and only if for each compact subspace $K$ of $\R_+$, the sequence $f_n|_K$ converges uniformly to $f|_K$.
\end{proposition}
The following corollary follows from Sion's theorem, \cite{sion1958}. 

\begin{corollary}[Proposition A.10 in \citep{Sor:Springer2002}] 
\label{cor_sion}
Let $(X,Y,g)$ be a zero-sum game, i.e., $X$ and $Y$ are the action sets of player 1 and 2 respectively and $g$ is the payoff to player 1, such that: X is a compact metric space, for all $y\in Y$, the function $g(\cdot,y)$ is upper semi-continuous. Then the game $(\Delta(X),\Delta_f(Y),g)$ has a value and player 1 has an optimal strategy.
\end{corollary}
We are now able to complete the proof of \cref{prop_existence_value}.
\begin{proof}[Proof of \cref{prop_existence_value}]
By Ascoli's theorem, see \citep[Theorem 47.1]{munkres2000}, $\W$ is compact for the topology of compact convergence.
Moreover, for all $(y,t)\in \A$ the function $g_{m,r}(\cdot,(y,t))$ is upper semi-continuous.
The conclusion follows from \cref{cor_sion}. 
\end{proof}

\subsection*{Omitted proofs of \cref{subsec_network}}

\begin{proof}[Proof of \cref{lem_tour}]
Let $\pi = (u_1,u_2,\dots,u_{n-1},u_n)$, $u_1=u_n=u\in V$ be an Eulerian tour. Without loss of generality, suppose $l(u_i,u_{i+1})\neq 0$ for all $i\in \{1,\dots,n-1\}$. The parametrization is constructed in the following way.

If $t\in [0,l(u_1,u_2)]$ then \[w(t) = \left(u_1,u_2,\frac{t}{l(u_1,u_2)}\right).\]
Else, suppose $n\geq 3$. For all $k\in \{2,\dots,n-1\}$ if \[t\in \left(\sum_{i=1}^{k-1} l(u_i,u_{i+1}),\sum_{i=1}^{k} l(u_i,u_{i+1})\right]\] then
\[w(t)=\left(u_k,u_{k+1},\frac{t-\sum_{i=1}^{k-1} l(u_i,u_{i+1})}{l(u_k,u_{k+1})}\right).\]
It is not difficult to verify that such $w$ is appropriate.
\end{proof}

\begin{proof}[Proof of \cref{th_euleriannetwork}]
If $m\geq \lambda(\n)$, the patroller guarantees 1 by playing a parametrization of $\n$. Suppose that $m<\lambda(\n)$. Let $(y,t)\in \n\times \R_+$ be a pure strategy of the attacker and let $w$ be as in \cref{def_param_network}. There exists $t_y\in [0,\lambda(\n)]$ such that $w(t_y) = y$. Now let $t_0\in [t_y-t-m,t_y-t]$. Then $w_{t_0}(t_y-t_0) = w(t_y) = y$. And $t_y-t_0 \in [t,t+m]$. Thus $y\in w_{t_0}([t,t+m]).$ Hence under the patroller's uniform strategy
\begin{align*}
\Proba(y\in w_{t_0}([t,t+m]))&\geq \Proba(t_0\in[t_y-t-m,t_y-t])=\frac{m}{\lambda(\n)}.
\end{align*}
The other inequality follows from \cref{proposition_upperbound} since in this case, $\rho$ equals $1$.
\end{proof}

\subsection*{Omitted proofs of \cref{subsec_simplesearch}}

To prove \cref{thm_simplesearchspace} we first need some preliminary definition and lemmas.

\begin{definition}
Let $Q$ be a search space. A continuous function $L:[0,1]\to Q$ such that $L(0)=L(1)$ is called an $r$-tour if for any $x\in Q$ there exists $l\in L([0,1])$ such that $d(x,l)\leq r$.
\end{definition}

The next lemma shows that when the radius of detection $r$ is small, one can find in $Q$ an $r$-tour with length not exceeding $\lambda(Q)/2r$, up to some $\varepsilon$.

\begin{lemma}[Lemma 3.39 in \citep{AlpGal:Kluwer2003}]
\label{lemma_rtour}
Let $Q\subset \R^2$ be a simple search space. Endow $Q$ with the Euclidean norm. Then for any $\varepsilon>0$ there exits $r_\varepsilon>0$ such that for any $r<r_\varepsilon$ there exists an $r$-tour $L:[0,1]\to Q$ such that \[TV(L) \leq (1+\varepsilon)\frac{\lambda(Q)}{2r}.\]
\end{lemma}

The next lemma gives a parametrization of $L([0,1])$ in terms of walks.

\begin{lemma}
\label{lem_parametrization}
Let $L$ be an $r$-tour as in \cref{lemma_rtour}. Then for all $\varepsilon'>0$ there exists $w : [0,TV(L)+\varepsilon'] \to L([0,1])$ continuous such that:
\begin{itemize}
\item[i)] $w(0)= w(TV(L)+\varepsilon')$;
\item[ii)] $w$ is surjective;
\item[iii)] $w$ is $1$-Lipschitz continuous;
\item[iv)] $TV(w)= TV(L)$.
\end{itemize}
$w$ is extended to a $(TV(L)+\varepsilon')$-periodic function on $\R$, which is still denoted $w$.
\end{lemma}

\begin{proof}[Proof of \cref{lem_parametrization}]
Let $\varepsilon'>0$ and let \[\begin{array}{ccccc}
f & : & [0,1] & \to & [0,TV(L)+\varepsilon'] \\
 & & s & \mapsto & TV(L|_{[0,s]})+\varepsilon' s.
\end{array}\] The function $f$ is increasing and continuous on $[0,1]$, hence $f$ is an homeomorphism. Define $w$ as $L \circ f^{-1}$ on $[0,TV(L)+\varepsilon']$.
It is not difficult to prove that such $w$ verifies the condition of the lemma.
\end{proof}

We are now able to prove \cref{thm_simplesearchspace}.

\begin{proof}[Proof of \cref{thm_simplesearchspace}]Let $L$ and $w$ be as in \cref{lemma_rtour,lem_parametrization} respectively. For all $t_0\in[0,TV(L)+\varepsilon']$ define $w_{t_0}(\cdot)$ as $w(t_0+\cdot)$.

Let $(l,t)\in L([0,1])\times \R_+$ . By \cref{lem_parametrization} ii), there exists $t_l\in [0,TV(L)+\varepsilon']$ such that $w(t_l) = l$.
Now let $t_0\in [t_l-t-m,t_l-t]$. Then $w_{t_0}(t_l-t_0) = w(t_l) = l$. And $t_l-t_0 \in [t,t+m]$. Hence $l\in w_{t_0}([t,t+m]).$

Suppose $t_0$ is chosen uniformly in $[0,TV(L)+\varepsilon']$. By \cref{lem_parametrization} iii) this is an admissible strategy for the patroller. Let $(y,t)\in \A$ be a pure strategy of the attacker. Then if $l\in L([0,1])$ is such that $d(y,l)\leq r$,
\begin{align*}
\Proba(d(y,w_{t_0}([t,t+m]))\leq r) &\geq \Proba(l \in w_{t_0}([t,t+m])),\\ 
&\geq \Proba(t_0 \in [t_l-t-m,t_l-t])\\
&= \frac{m}{TV(L)+\varepsilon'}.
\end{align*}
By \cref{lemma_rtour}, this last quantity is greater than or equal to $\frac{m}{\frac{(1+\varepsilon)\lambda(Q)}{2r}+\varepsilon'}$.
Hence the patroller guarantees $\frac{m}{\frac{(1+\varepsilon)\lambda(Q)}{2r}+\varepsilon'}$ for all $\varepsilon'>0$, that is \[V_Q(m,r)\geq \frac{2rm}{(1+\varepsilon)\lambda(Q)} \sim \frac{2rm}{\lambda(Q)}\] as $r$ goes to $0$.

In this context, \cref{proposition_upperbound} yields $V_Q(m,r)\leq \frac{2rm+\pi r^2}{\lambda(Q)}\sim \frac{2rm}{\lambda(Q)}$ as $r$ goes to $0$.
\end{proof}

\subsection*{Omitted proofs of \cref{subsec_asymptotichinding}}

To prove \cref{theorem_equivalent_static} we first need to introduce a technical lemma.

Denote $B^2_r(x) = \{y\in \R^n \ | \ \eucl{x - y} \leq r \}$ the closed ball of center $x$ with radius $r$ for the Euclidean norm, and $\partial B^2_r(x) = \{y\in \R^n \ | \ \eucl{x - y} = r \}$ the sphere of center $x$ with radius $r$ for the Eucliean norm.

The intuition behind \cref{lemma_intersection} below is the following. We consider the balls $B^2_\varepsilon(0)$ and $B^2_r(x)$ with $x$ on the boundary of $B^2_\varepsilon(0)$. When $r$ goes to zero, the ratio between the volume of the ball $B^2_r(x)$ and the ball $B^2_r(x)$ intersected with the ball $B^2_\varepsilon(0)$ goes to $2$. \cref{lemma_intersection} gives an upper bound to this ratio, as $r$ goes to $0$,  for a non necessary Euclidean ball $B_r(x)$.

\begin{lemma}
\label{lemma_intersection}
Let $\norm{\cdot}$ be a norm on $\R^n$ and $c_1,c_2>0$ be such that $c_1\norm{\cdot} \leq \eucl{\cdot} \leq c_2\norm{\cdot}$. Then for all $x\in \partial B^2_{\varepsilon}(0)$
\[\limsup_{r\to 0} \frac{\lambda(B_r)}{\lambda(B^2_{\varepsilon}(0)\cap B_r(x))} \leq 2\left(\frac{c_2}{c_1}\right)^n.\]
\end{lemma}

\begin{proof}[Proof of \cref{lemma_intersection}]
Let $x\in \partial B_\varepsilon^2(0)$, let $\varepsilon>0$.
Denote $I$ the regularized incomplete Beta function: for $a,b>0$ and $0<z<1$, $I_z(a,b) = \frac{B(z;a,b)}{B(a,b)}$. Where $B(z;a,b) = \int_0^z t^{a-1}(1-t)^{b-1}dt$ and $B(a,b) = B(1;a,b)$ is the Beta function.
Then we have, see \cite{li2011},
\begin{align*}
\lambda\left(B^2_{\varepsilon}(0)\cap B^2_r(x)\right) = \frac{\pi^{n/2}}{2\Gamma(\frac{n}{2}+1)}&\Bigg(r^n I_{1-\left(\frac{r}{2\varepsilon}\right)^2}\left(\frac{n+1}{2},\frac{1}{2}\right) \\
&+ \varepsilon^n I_{\left(\frac{r}{\varepsilon}\right)^2\left(1-\left(\frac{r}{2\varepsilon}\right)^2\right)}\left(\frac{n+1}{2},\frac{1}{2}\right)\Bigg).
\end{align*}
Since $t\mapsto t^{\frac{n-1}{2}}(1-t)^{-1/2}$ is integrable over $[0,1)$,
\begin{align*}
I_{1-\left(\frac{r}{2\varepsilon}\right)^2}\left(\frac{n+1}{2},\frac{1}{2}\right) = \frac{\int_0^{1-\left(\frac{r}{2\varepsilon}\right)^2}t^\frac{n-1}{2}(1-t)^{-1/2}dt}{B(\frac{n-1}{2},\frac{1}{2})} \to 1,
\end{align*}
as $r$ goes to $0$.
And,
\begin{align*}
I_{\left(\frac{r}{\varepsilon}\right)^2\left(1-\left(\frac{r}{2\varepsilon}\right)^2\right)}\left(\frac{n+1}{2},\frac{1}{2}\right) &= \frac{\int_0^{\left(\frac{r}{\varepsilon}\right)^2\left(1-\left(\frac{r}{2\varepsilon}\right)^2\right)} t^\frac{n-1}{2}(1-t)^{-1/2}dt}{B(\frac{n-1}{2},\frac{1}{2})}
\end{align*}
which, since $1 \leq (1-t)^{-1/2}$ when $t\in[0,1)$, is greater than
\begin{align*}
\frac{\frac{2}{n+1}\left(\frac{r}{\varepsilon}\right)^{n+1}\left(1-\left(\frac{r}{2\varepsilon}\right)^2\right)^\frac{n+1}{2}}{B(\frac{n-1}{2},\frac{1}{2})} = \frac{2r^{n+1}}{(n+1)\varepsilon^{n+1} B(\frac{n-1}{2},\frac{1}{2})} + o(r^{2n+2})
\end{align*}
when $r$ goes to $0$. Hence we have
\[\lambda\left(B^2_{\varepsilon}(0)\cap B^2_r(x)\right) \geq \frac{\pi^{n/2}}{2\Gamma(\frac{n}{2}+1)}(r^n+o(r^n))\]
as $r$ goes to $0$. Moreover since
\begin{align*}
B^2_{\varepsilon}(0)\cap B_r(x) &= \{y\in \R^n \ | \ \eucl{y}\leq \varepsilon \text{ and } \norm{x-y} \leq r\} \\
&\supset \{y\in \R^n \ | \ \eucl{y}\leq \varepsilon \text{ and } \eucl{x-y} \leq c_1r\},
\end{align*}
and $B_r(0)\subset B^2_{c_2r}(0)$,
we have $\lambda\left(B^2_{\varepsilon}(0)\cap B_r(x)\right) \geq \lambda\left(B^2_{\varepsilon}(0)\cap B^2_{c_1r}(x)\right),$
and $
c_2^n\lambda(B^2_r)\geq \lambda(B_r).$
Finally, dividing by $\lambda\left(B^2_{\varepsilon}(0)\cap B_r(x)\right)$ and taking the $\limsup$, since $\lambda(B_r^2) = \frac{\pi^{n/2}r^n}{\Gamma\left(\frac{n}{2}+1\right)}$ we have
\begin{align*}
\limsup_{r\to 0}\frac{\lambda(B_r)}{\lambda\left(B^2_{\varepsilon}(0)\cap B_r(x)\right)} \leq \limsup_{r\to 0}\frac{c_2^n\lambda(B^2_r)}{\lambda\left(B^2_{\varepsilon}(0)\cap B^2_{c_1r}(x)\right)} \leq 2\left(\frac{c_2}{c_1}\right)^n. 
\end{align*}
\end{proof}

We are now able to prove \cref{theorem_equivalent_static}.

\begin{proof}[Proof of \cref{theorem_equivalent_static}]
Let $\varepsilon >0$ and $r\in (0,\varepsilon)$. We regularize the boundary of $Q$ by defining $Q_\varepsilon = Q + B^2_\varepsilon(0),$ and $I^\varepsilon (r) = \{y \in Q_\varepsilon \ | \ B_r(y) \subset Q_\varepsilon \}.$
Define as well $\lambda^\varepsilon_{\min} (r) = \min_{y\in Q_\varepsilon} \lambda (B_r(y)\cap Q_\varepsilon).$
Finally define $\mu \in \Delta(Q_\varepsilon)$ such that for all $B\subset Q_\varepsilon$ measurable 
\[\mu(B) = \frac{\lambda\left(B\cap I^\varepsilon (r)\right)\lambda^\varepsilon_{\min}(r) + \lambda(B\cap (Q_\varepsilon \setminus I^\varepsilon (r))\lambda(B_r)}{\lambda(I^\varepsilon(r))\lambda^\varepsilon_{\min}(r)+\lambda(B_r)\lambda(Q_\varepsilon\setminus I^\varepsilon(r))}.\]
Since by definition $\lambda(B_r)\geq \lambda^\varepsilon_{\min} (r)$, for all $x\in Q_\varepsilon$
\[\mu(B_r(x)\cap Q_\varepsilon)\geq \frac{\lambda^\varepsilon_{\min}(r)\lambda(B_r)}{\lambda(I^\varepsilon(r))\lambda^\varepsilon_{\min}(r)+\lambda(B_r)\lambda(Q_\varepsilon\setminus I^\varepsilon(r))}.\]
Because the hider can play in $(Q_\varepsilon,r)$ as he would play in $(Q,r)$, $V_{Q_\varepsilon}(r) \leq V_Q(r)$.
By \cref{proposition_upperbound},
\[\frac{\lambda^\varepsilon_{\min}(r)\lambda(B_r)}{\lambda(I^\varepsilon(r))\lambda^\varepsilon_{\min}(r)+\lambda(B_r)\lambda(Q_\varepsilon\setminus I^\varepsilon(r))}\leq V_{Q_\varepsilon}(r) \leq V_Q(r) \leq \frac{\lambda(B_r)}{\lambda(Q)}.\]
Dividing by $\lambda(B_r)/\lambda(Q),$
\begin{align}
\label{eq_lambdaepsiloneq}
\frac{\lambda^\varepsilon_{\min}(r)\lambda(Q)}{\lambda(I^\varepsilon(r))\lambda^\varepsilon_{\min}(r)+\lambda(B_r)\lambda(Q_\varepsilon\setminus I^\varepsilon(r))} \leq \frac{V_Q(r)\lambda(Q)}{\lambda(B_r)} \leq 1.
\end{align}

Let us show that for all
$\varepsilon >0 \ \bigcup_{r>0}I^\varepsilon (r) = \mathring{Q_\varepsilon}.$
Indeed, let $y\in \bigcup_{r>0}I^\varepsilon (r)$. There exists $r>0$ such that $y\in I^\varepsilon (r)$. Thus there exists $r>0$ such that $B_r(y) \subset Q_\varepsilon$.
Conversely, let $y\in \mathring{Q_\varepsilon}$. There exists $r'>0$ such that $B'_{r'}(y)\subset \mathring{Q_\varepsilon}$, where $B'_{r'}(y) = \{x\in \R^n \ | \ \Vert x - y\Vert < r'\}$. Take $0<r<r'$, then $B_r(y)\subset \mathring{Q_\varepsilon}$ hence $y\in I^\varepsilon(r)$.

For all $r_1,r_2>0$ such that $r_1>r_2$ one has $I^\varepsilon(r_1)\subset I^\varepsilon(r_2)$.
Hence $\lim_{r \to 0}\lambda(I^\varepsilon(r)) = \lambda(\mathring{Q_\varepsilon})$.
Dividing by $\lambda^\varepsilon_{\min}(r)$ and letting $r$ go to $0$ in \cref{eq_lambdaepsiloneq}, by \cref{lemma_intersection} one has, since the minimum in $\lambda^\varepsilon_{\min}(r)$ is reached on the boundary of a Euclidean ball,
\begin{align}
\label{eq_c1c2}
\frac{\lambda(Q)}{\lambda(\mathring{Q_\varepsilon})+2\left(\frac{c_2}{c_1}\right)^n\lambda(\partial Q_\varepsilon)} \leq \liminf\limits_{r\to 0} \frac{V_Q(r)\lambda(Q)}{\lambda(B_r)} \leq \limsup\limits_{r\to 0} \frac{V_Q(r)\lambda(Q)}{\lambda(B_r)} \leq 1.
\end{align}

Let us show that $\bigcap_{\varepsilon>0} \mathring{Q_\varepsilon} = \bigcap_{\varepsilon>0} Q_\varepsilon = Q.$
Indeed, let $y\in \bigcap_{\varepsilon > 0} Q_\varepsilon$. For all $\varepsilon > 0$ $\min_{z\in Q} \eucl{y-z}\leq \varepsilon$, hence $y\in Q$.
Conversely, for all $\varepsilon >0$ $Q\subset \mathring{Q_\varepsilon}$ hence $Q \subset \bigcap_{\varepsilon>0} \mathring{Q_\varepsilon}$.
Moreover for all $\varepsilon_1, \varepsilon_2>0$ such that $\varepsilon_1 < \varepsilon_2$ one has $Q_{\varepsilon_1} \subset Q_{\varepsilon_2}$.
Hence $\lim_{\varepsilon\to 0}\lambda(\mathring{Q_\varepsilon})= \lambda(Q)$, $\lim_{\varepsilon\to 0} \lambda(Q_\varepsilon) = \lambda(Q)$ and $\lambda(\partial Q_\varepsilon) = \lambda(Q_\varepsilon) - \lambda(\mathring{Q_\varepsilon})$ so $\lim_{\varepsilon \to 0} \lambda(\partial Q_\varepsilon) = 0.$

Letting $\varepsilon \to 0$ in \cref{eq_c1c2},
$1=\frac{\lambda(Q)}{\lambda(Q)} \leq \lim_{r\to 0}\frac{V_Q(r)\lambda(Q)}{\lambda(B_r)} \leq 1.$
\end{proof}

\section*{Acknowledgments}
The author wishes to express his gratitude to his Ph.D. advisor Jérôme Renault, as well as Marco Scarsini for his help in improving the general presentation of the paper.


\bibliographystyle{apalike}
\bibliography{bib_total}
\end{document}